%% file: main.tex
\documentclass[12pt,letterpaper]{article}
\pdfoutput=1

\usepackage{authblk}
\usepackage{amsmath,amsthm,amsfonts,amssymb,amscd}
\usepackage{bbm}
\usepackage{easybmat}
\usepackage[shortlabels]{enumitem}
\usepackage{fancyhdr}
\usepackage{fullpage}
\usepackage[margin=3cm,bottom=4cm,marginparwidth=2cm]{geometry}
\usepackage{graphicx}
\usepackage{lastpage}
\usepackage{mathrsfs}
\usepackage{mathtools}
\usepackage[numbers]{natbib}
\usepackage{nicematrix}
\usepackage[outline]{contour}
\usepackage{tikz-cd}
\usepackage{tikz-3dplot}
\usepackage{todonotes}
\usepackage{xfrac}

\usepackage{hyperref}
\usepackage[capitalise,nameinlink]{cleveref}
\hypersetup{
  colorlinks   = true, 
  urlcolor     = blue, 
  linkcolor    = blue, 
  citecolor    = black 
}

\usetikzlibrary{3d}
\NiceMatrixOptions{cell-space-limits = 1pt}

\input{symbol_defns.tex}

\renewcommand{\epsilon}{\varepsilon}

\theoremstyle{definition}
\newtheorem{definition}{Definition}[section]
\newtheorem{example}[definition]{Example}

\theoremstyle{plain}
\newtheorem{theorem}[definition]{Theorem}
\newtheorem{lemma}[definition]{Lemma}
\newtheorem{proposition}[definition]{Proposition}
\newtheorem{corollary}[definition]{Corollary}
\newtheorem*{corollary*}{Corollary}

\title{Real Stability and Log Concavity are coNP-Hard}
\author[1]{Tracy Chin}
\affil[1]{Department of Mathematics, University of Washington, \href{mailto:tlchin@uw.edu}{tlchin@uw.edu}}
\date{\today}

\begin{document}
\maketitle

\begin{abstract}
    Real-stable, Lorentzian, and log-concave polynomials are well-studied classes of polynomials, and have been powerful tools in resolving several conjectures. We show that the problems of deciding whether a polynomial of fixed degree is real stable or log concave are coNP-hard. On the other hand, while all homogeneous real-stable polynomials are Lorentzian and all Lorentzian polynomials are log concave on the positive orthant, the problem of deciding whether a polynomial of fixed degree is Lorentzian can be solved in polynomial time.
\end{abstract}

\section{Introduction}
The theory of stable polynomials has been a key ingredient in seemingly unrelated mathematical advances. In the past few years, they were used to construct infinite families of Ramanujan graphs \cite{Marcus2015ramanujan, Marcus2018ramanujan}, and to resolve the Kadison-Singer problem \cite{Marcus2015kadisonsinger}.

Hyperbolic polynomials are a closely related class. They first arose in PDEs \cite{Garding1951pde}, but have also gained traction in the optimization community, for example as barrier functions for interior point methods \cite{guler1997interiorpoint}.

Beyond hyperbolic polynomials are the classes of Lorentzian and log-concave polynomials. Lorentzian polynomials in particular have close ties to combinatorial Hodge theory \cite{Huh2023hodge} and matroids \cite{anari2018logconcave3,brändén2022lorentzian}. They are also called completely log-concave polynomials, and they provide a bridge between continuous log concavity of a polynomial and discrete log concavity of its coefficients \cite{brändén2022lorentzian}. Because of this connection to discrete log concavity, they have been used to prove log concavity of several famous combinatorial sequences, including Kostka numbers \cite{Huh2022Schur} and the coefficients of the Alexander polynomial for certain types of links \cite{Hafner2023Alexander}.

While stable, completely log-concave, and log-concave polynomials are powerful tools, less is known about how to test if a given polynomial has these properties. In \cite{raghavendra2016real}, \citeauthor{raghavendra2016real} give a polynomial time algorithm for checking real stability of bivariate polynomials, but their techniques do not generalize easily to more variables. Instead, the results using stable polynomials in \cite{Marcus2015kadisonsinger,Marcus2015ramanujan,Marcus2018ramanujan} often start with a known construction for stable polynomials, then use a series of stability-preserving operations to reach the desired polynomial, rather than checking stability directly. It is relatively easy to generate particular stable polynomials, and \citeauthor{borcea2009stabilitypreserver} give a full classification of stability-preserving operations \cite{borcea2009stabilitypreserver}, enabling this style of argument for proving particular polynomials are real stable. Similarly, to prove Lorentzianity in \cite{Hafner2023Alexander}, \citeauthor{Hafner2023Alexander} construct a multivariate generalization of the Alexander polynomial in order to show that the single variable specialization has log-concave coefficients.

In this paper, we study the hierarchy of stable, completely log-concave, and log-concave polynomials through the lens of computational complexity. In previous works, \citeauthor{saunderson2019certifying} showed that it is coNP-hard to decide if a polynomial is hyperbolic with respect to a given direction \cite{saunderson2019certifying}, and \citeauthor{Ahmadi_2011} showed that it is NP-hard to test if a polynomial is convex \cite{Ahmadi_2011}. We extend and refine the techniques and results from these papers to get computational complexity results on testing real stability and log concavity.

In \cref{sec:stability}, we show that it is coNP-hard to determine if a homogeneous polynomial is real stable. The main result is \cref{thm:cubic-stable-hard}, which states that it is coNP-hard to decide if a homogeneous cubic with rational coefficients is real stable. As an easy consequence, we then get \cref{cor:stability-conp-hard}, which states that deciding if a homogeneous polynomial is real stable is coNP-hard in all degrees $d \geq 3$.

By contrast, deciding whether a polynomial is completely log concave can be done in polynomial time (\cref{sec:complete-log concavity}). Specifically, in \cref{prop:clc-poly-time} we give an explicit algorithm to decide if a homogeneous polynomial of degree $d$ is completely log concave that runs in time $O(n^{d+1})$.

On the other hand, log concavity is again coNP-hard (\cref{sec:quartic-log-concavity}). In \cref{thm:quartic-lc-hard}, we show that it is coNP-hard to decide if a homogeneous polynomial of degree four is log concave on $\RR^n_{\geq 0}$, and we generalize this result to all degrees $\geq 4$ in \cref{cor:lc-hard-higher-deg}.

Finally, in \cref{sec:log concavity-cubic}, we study log concavity properties of cubic polynomials. Specifically, we show that we can decide if a homogeneous cubic polynomial is log concave in polynomial time, but that the weaker property of directional log concavity is coNP-hard.

\section{Preliminaries}
\subsection{Notation}
Throughout this paper, we use $\RR_{\geq 0}[x_1,\dots,x_n]$ to denote the set of polynomials in variables $x_1,\dots,x_n$ with nonnegative coefficients. We use $\RR[x_1,\dots,x_n]_d$ to denote the set of polynomials with real coefficients that are homogeneous of degree $d$.

If $f\in\RR[x_1,\dots,x_n]$, we use the shorthand $\partial_i f \coloneq \frac{\dl}{\dl x_i}f$ and analogously $\dl_i^k f \coloneq \frac{\dl^k}{\dl x_i^k}f$. For $\alpha = (\alpha_1,\dots,\alpha_n) \in \NN^n$, we use the shorthand $\dl^\alpha f \coloneq \dl_1^{\alpha_1}\dots\dl_n^{\alpha_n} f$.

\subsection{Log-Concave and Lorentzian Polynomials}
\begin{definition}\label{def:log concavity}
    A polynomial $g \in \RR_{\geq 0}[x_1,\dots,x_n]$ is \textit{log concave} if $\log(g)$ is a concave function over $\RR^n_{> 0}$. Equivalently, $g$ is log concave if and only if $\nabla^2\log(g) \preceq 0$ at all points of $\RR^n_{\geq 0}$ where it is defined.
\end{definition}

For homogeneous polynomials of degree $d \geq 2$, one can also characterize log concavity using the Hessian of $f$.
\begin{theorem}[\cite{brändén2022lorentzian}, Proposition 2.33]\label{thm:log concave-hessian}
    If $f$ is a homogeneous polynomial in $n \geq 2$ variables of degree $d\geq 2$, then the following are equivalent for any $w\in\RR^n$ satisfying $f(w) > 0$.
    \begin{enumerate}[(1)]
        \item The Hessian of $f^{1/d}$ is negative semidefinite at $w$.
        \item The Hessian of $\log f$ is negative semidefinite at $w$.
        \item The Hessian of $f$ has exactly one positive eigenvalue at $w$.
    \end{enumerate}
\end{theorem}

Complete log concavity is a stronger condition that requires that both the polynomial and all its directional derivatives be log concave.
\begin{definition}[\cite{anari2018logconcave}]\label{def:completely-log concave}
    A polynomial $g\in\RR[x_1,\dots,x_n]$ is \textit{completely log concave} if for every $k \geq 0$ and nonnegative matrix $V \in \RR^{n\times k}_{\geq 0}$, $D_Vg(z)$ is nonnegative and log concave as a function over $\RR^n_{> 0}$, where
    \[D_Vg(z) = \left(\prod_{j=1}^k\sum_{i=1}^n V_{ij}\partial_i\right)g(z).\]
\end{definition}
Homogeneous completely log-concave polynomials are also called Lorentzian polynomials. The equivalence of the definitions is proved in \cite[Theorem 2.30]{brändén2022lorentzian}. If $f$ is multiaffine, then log concavity and complete log concavity coincide, and $f$ is log concave if and only if it is completely log concave.

\subsection{Hyperbolic Polynomials}
\begin{definition}
    A homogeneous polynomial $p \in \RR[x_1,\dots,x_n]_d$ is \textit{hyperbolic with respect to $e\in\RR^n$} if $p(e) > 0$ and for all $x \in \RR^n$, the univariate polynomial $p(x+te) \in \RR[t]$ is real rooted.
\end{definition}

If $p$ is hyperbolic with respect to $e$, then we define the associated \textit{hyperbolicity cone}
\[\Lambda_+(p,e) = \{x\in\RR^n : \text{all roots of } p(te-x) \text{ are positive}\}.\]

The following is a classic result on hyperbolic polynomials, due to \citeauthor{garding1959hyperbolic} \cite{garding1959hyperbolic}.

\begin{theorem}[\cite{garding1959hyperbolic}, Theorem 2]\label{thm:garding-hyperbolicity-cone}
    If $p$ is hyperbolic with respect to $e$, then it is also hyperbolic with respect to $x$ for any $x \in \Lambda_+(p,e)$. Moreover, the closure of the hyperbolicity cone $\overline{\Lambda_+}(p,e)$ can be described explicitly as the closure of the connected component of $\RR^n \setminus \{x\in\RR^n : p(x) = 0\}$ containing $e$.
\end{theorem}

\begin{example}
    Let $p(x_0,x_1,\dots,x_n) = x_0^2 - \sum_{i=1}^n x_i^2$. Then $p$ is hyperbolic with respect to $e = (1,0,\dots,0)$. Indeed, for any $w\in\RR^{n+1}$, then the univariate restriction $p(w+te)$ is
    \[p(w+te) = (t+w_0)^2 - \sum_{i=1}^nw_i^2.\]
    Since $\sum_{i=1}^n w_i^2 \geq 0$, this polynomial is real rooted, so $p$ is hyperbolic with respect to $e$.

    In this example, the hyperbolicity cone is
    \[
        \Lambda_+(p,e) = \{(x_0,x)\in\RR^{n+1} : x_0 > \lVert x\rVert_2\}.
    \]
    Indeed,
    \[p(te - (x_0,x)) = (t-x_0)^2 - \sum_{i=1}^n x_i^2,\]
    which has roots $t = x_0 \pm \sqrt{\sum_{i=1}^n x_i^2} = x_0 \pm \lVert x\rVert_2,$ so both roots are positive if and only if $x_0 > \lVert x\rVert_2$. This agrees with the characterization from \cref{thm:garding-hyperbolicity-cone}, since $p(x_0,x) = 0$ if and only if $x_0^2 = \lVert x\rVert^2$, and $e=(1,0,\dots,0)$ is in the connected component where $x_0 > \lVert x \rVert$.
\end{example}

Throughout this paper, when we say $p$ is hyperbolic with respect to a cone $K$, it means that $p$ is hyperbolic with respect to every $x\in K$.

\subsection{Stable Polynomials}
\begin{definition}
    A polynomial $p \in \CC[x_1,\dots,x_n]$ is \textit{stable} if $p(z_1,\dots,z_n) \neq 0$ whenever $z_1,\dots,z_n$ all have strictly positive imaginary part. If additionally $p$ has real coefficients, then we say that $p$ is \textit{real stable}.
\end{definition}

A homogeneous polynomial is real stable if and only if it is hyperbolic with respect to $\RR^n_{>0}$. More generally, a polynomial $p\in\RR[x_1,\dots,x_n]$ is real stable if and only if its homogenization $p^{\hom}(x_0,x_1,\dots,x_n) = x_0^{\deg(p)}p(\frac{x_1}{x_0},\dots,\frac{x_n}{x_0})$ is hyperbolic with respect to the cone $\{0\}\times\RR^{n}_{>0}\subseteq \RR^{n+1}$ \cite[Proposition 1.1]{borceabranden2010Weyl}. For the sake of completeness, we reproduce the proof of the result in the homogeneous case here.
\begin{proposition}
    Let $f\in\RR[x_1,\dots,x_n]$ be a homogeneous polynomial of degree $d$. Then $f$ is real stable if and only if it is hyperbolic with respect to $\RR^n_{>0}$.
\end{proposition}
\begin{proof}
    First, suppose that $f$ is not real stable. Then there exists some $z\in\CC^n$ with $\Im(z_1),\dots,\Im(z_n)>0$ such that $f(z) = 0$. Let $v = \Im(z) \in \RR^n_{>0}$ and $x = \Re(z) \in \RR^n$. Then $f(iv + x) = f(z) = 0$, so the univariate polynomial $f(tv + x)$ has a root at $t = i$. Thus, $f(tv+x)$ is not real rooted, so $f$ is not hyperbolic with respect to $v\in\RR^n_{>0}$.

    Now, suppose that $f$ is not hyperbolic with respect to some $v \in \RR^n_{>0}$. Then there exists $x\in\RR^n$ such that $f(tv + x)$ has a root $t_0 = a + bi$ with $b \neq 0$. Since $f(tv+x)$ has real coefficients, then the roots of $f(tv+x)$ come in conjugate pairs, so we may assume without loss of generality that $b > 0$. Let $z_i = t_0v_i + x_i$, so $f(z_1,\dots,z_n) = 0$ and $\Im(z_i) = bv_i$. Since $b > 0$ and $v\in\RR^n_{>0}$, then $\Im(z_i) > 0$ for all $i$, so $f$ is not stable.
\end{proof}

\subsection{Hierarchy of Polynomials}
Let $p \in \RR_{\geq 0}[x_1,\dots,x_n]_d$ be a homogeneous polynomial. As we saw above, $p$ is real stable if and only if it is hyperbolic with respect to $\RR^n_{> 0}$. Moreover, if $p$ is hyperbolic with respect to some cone $K$, then it is also completely log concave with respect to $K$ \cite[Theorem 11]{dey2023polynomials}. Finally, if $p$ is completely log concave, then it is by definition also log concave by taking $k = 0$ in \cref{def:completely-log concave}.

When $d=2$, this hierarchy collapses, and $p$ is stable if and only if it is log concave. Also, we show in \cref{sec:homogeneous-cubics} that if $d = 3$, then $p$ is log concave if and only if it is completely log concave. However, for $d \geq 4$, all of the containments are strict.

\section{Computational Complexity of Stability}\label{sec:stability}
In \cite{raghavendra2016real}, \citeauthor{raghavendra2016real} give a deterministic algorithm that runs in time $O(d^5)$ to determine if a bivariate polynomial $p\in\RR[x,y]$ of degree $d$ is real stable. They also ask whether their algorithm can be generalized to three or more variables.

In this section, we show that it is coNP-hard to test whether a homogeneous cubic with rational coefficients is real stable. To do this, we will leverage the following result from \cite{saunderson2019certifying}:
\begin{theorem}[\cite{saunderson2019certifying}, Proposition 5.1]\label{thm:hyp-iff-max-le-1}
    If $p(x_0,x) = x_0^3 - 3x_0\lVert x\rVert^2 + 2q(x)$, where $q \in \RR[x_1,\dots,x_n]_3$, then $p$ is hyperbolic with respect to $e_0$ if and only if $\max_{\lVert x\rVert^2 = 1} |q(x)| \leq 1$.
\end{theorem}
In \cite{saunderson2019certifying}, \citeauthor{saunderson2019certifying} proves this result and uses it to show that it coNP-hard to decide if a homogeneous cubic polynomial is hyperbolic in a given direction. In this section, we use similar techniques to show that stability is also coNP-hard.

To apply this result, we will construct a polyhedral cone $K$ such that $p$ is hyperbolic with respect to $e_0$ if and only if it is hyperbolic with respect to $K$. After an appropriate change of variables, this is equivalent to being stable. 

For $\epsilon > 0$, let $L_\epsilon \subseteq\RR^{n+1}$ denote the Lorentz cone
\[
    L_\epsilon = \{(x_0,\vec{x}) \in \RR^{n+1} : x_0 \geq \frac{1}{\epsilon}\lVert x\rVert\}.
\]

\begin{lemma}\label{lem:nonneg-on-lorentz-cone}
    Let $p(x_0,x) = x_0^3 - 3x_0 \lVert x\rVert^2 + 2q(x)$, where $q(x) \in \RR[x_1,\dots,x_n]_3$ is a homogeneous cubic. Let $N$ be the largest absolute value of a coefficient of $q$, and let $0 < \epsilon < \min(\frac{1}{2Nn^3}, \frac{1}{2})$. Then $p(\lVert x\rVert, \epsilon x) > 0$ for all $x \in \RR^n\setminus\{0\}$.
\end{lemma}
\begin{proof}
    By homogeneity,
    \begin{align*}
        p(\lVert x\rVert, \epsilon x) &= \lVert x\rVert^3 -3 \lVert x\rVert(\epsilon^2 \lVert x\rVert^2) + 2q(\epsilon x)\\
            &= \lVert x\rVert^3\left(1 - 3\epsilon^2 + 2\epsilon^3 q\left(x/\lVert x\rVert\right)\right).
    \end{align*}
    Therefore, $p(\lVert x\rVert, \epsilon x) > 0$ if and only if $1-3\epsilon^2 + 2\epsilon^3 q(\frac{x}{\lVert x\rVert}) > 0$, and so $p(\lVert x\rVert, \epsilon x) > 0$ for all $x \in \RR^n\setminus\{0\}$ if and only if
    \[1 - 3\epsilon^2 + 2\epsilon^3\left(\min_{\lVert x\rVert = 1} q(x)\right) >0.\]

    Since $q$ is homogeneous of degree three, $\min_{\lVert x \rVert = 1} q(x) = -\max_{\lVert x\rVert = 1} q(x)$, so $p(\lVert x\rVert, \epsilon x) > 0$ for all $x \neq 0$ if and only if
    \[1 - 3 \epsilon^2 - 2\epsilon^3 \left(\max_{\lVert x\rVert = 1} q(x)\right) > 0.\]

    Since $q$ is homogeneous of degree three, it has at most $\binom{n+2}{3} \leq n^3$ terms. Moreover, if $\lVert x\rVert = 1$, each monomial has absolute value at most one, so each term is $\leq N$. Thus, $\max_{\lVert x\rVert = 1} q(x) \leq Nn^3$.

    Therefore,
    \[
        1 - 3\epsilon^2 - 2\epsilon^3 \left(\max_{\lVert x\rVert = 1} q(x)\right) \geq 1-3\epsilon^2 - 2\epsilon^3 Nn^3.
    \]
    Since $\epsilon < \frac{1}{2Nn^3}$ and $\epsilon < \frac{1}{2}$,
    \[1-3\epsilon^2 - 2\epsilon^3 Nn^3 > 1 - 3\epsilon^2 -2\epsilon^2\frac{1}{2Nn^3}(Nn^3) = 1-4\epsilon^2 > 0.\]

    Therefore, for any $0 < \epsilon < \min(\frac{1}{2Nn^3},\frac{1}{2})$, $p(\lVert x\rVert, \epsilon x) > 0$ for all $x \in \RR^n\setminus\{0\}$.
\end{proof}

\begin{proposition}\label{prop:hyp-wrt-Lorentz}
    Let $p(x_0,x) = x_0^3 - 3x_0 \lVert x\rVert^2 + 2q(x)$, where $q(x)$ is a homogeneous polynomial of degree three. Let $N$ be the largest absolute value of a coefficient of $q$, and let $0 < \epsilon < \min(\frac{1}{2Nn^3}, \frac{1}{2})$. Then $p$ is hyperbolic with respect to $e_0$ if and only if it is hyperbolic with respect to $L_{\epsilon}$.
\end{proposition}
\begin{proof}
    For one direction, we note that $e_0 \in L_{\epsilon}$. Thus, if $p$ is hyperbolic with respect to $L_{\epsilon}$, then in particular it is hyperbolic with respect to $e_0$.

    Now, suppose that $p$ is hyperbolic with respect to $e_0$. Since the hyperbolicity cone $\Lambda_+(p,e_0)$ is exactly the connected component of $\{(x_0,x)\in\RR^{n+1} : p(x_0,x) \neq 0\}$ containing $e_0$, it suffices to show that $p(x_0,x) > 0$ on $L_{\epsilon}$.

    Let $(x_0,x) \in L_{\epsilon}$, so that $x_0 = \frac{1}{\delta}\lVert x\rVert \geq \frac{1}{\epsilon}\lVert x\rVert$ for some $0 < \delta \leq \epsilon$. Then in particular, $0 < \delta < \min(\frac{1}{2Nn^3},\frac{1}{2})$, so by \cref{lem:nonneg-on-lorentz-cone},
    \[
        p(x_0,x) = p(\tfrac{1}{\delta} \lVert x\rVert, x) = p\left(\left\lVert \frac{1}{\delta}x\right\rVert, \delta\left(\frac{1}{\delta}x\right)\right) > 0.
    \]
    Therefore, $p(x_0, x) > 0$ for all $(x_0,x) \in L_{\epsilon}$, so $L_{\epsilon}$ is contained in the connected component of $\RR^{n+1}\setminus \{x\in\RR^{n+1} : p(x) = 0\}$ containing $e_0$. Thus, if $p$ is hyperbolic with respect to $e_0$, then it is also hyperbolic with respect to $L_{\epsilon}$.
\end{proof}

\begin{figure}
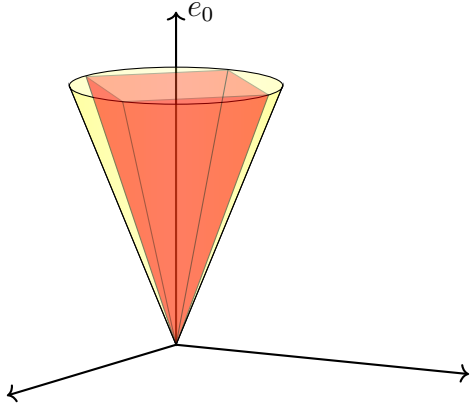

\centering
\include{lorentz_cone}
\caption{Cones $L_\epsilon \supseteq K$, both containing $e_0$. These cones are constructed so that $p > 0$ on $L_\epsilon$ and $K$, so $p$ is hyperbolic with respect to $e_0$ if and only if it is hyperbolic with respect to $K$.}
\end{figure}

\begin{corollary}\label{cor:hyp-wrt-simplicial}
    Let $p$ and $\epsilon$ be as above, and let $K = \cone\{e_0 \pm \epsilon e_i : i = 1,\dots, n\}$. Then $p$ is hyperbolic with respect to $e_0$ if and only if $p$ is hyperbolic with respect to $K$.
\end{corollary}
\begin{proof}
    Since $e_0 \in K$, the if direction is immediate.
    
    Conversely, by \cref{prop:hyp-wrt-Lorentz}, if $p$ is hyperbolic with respect to $e_0$, then it is hyperbolic with respect to $L_\epsilon$. Since $K \subseteq L_\epsilon$, it follows that $p$ is hyperbolic with respect to $K$.
\end{proof}

\begin{theorem}\label{thm:hyp-iff-stable}
    Let $p(x_0,x) = x_0^3 -3x_0\lVert x\rVert^2 + 2q(x)$, where $q \in\RR[x_1,\dots,x_n]$ is homogeneous of degree three. Let $N$ be the largest absolute value of a coefficient of $q$, let $0<\epsilon < \min(\frac{1}{2Nn^3}, \frac{1}{2})$, and let $M$ be the $(n+1)\times 2n$ matrix whose columns are $e_0 \pm \epsilon e_i$ for $i = 1,\dots, n$. Define
    \[\til{p}(x_1,\dots,x_{2n}) = p(Mx) \in \RR[x_1,\dots,x_{2n}]_3.\]
    Then $p$ is hyperbolic with respect to $e_0$ if and only if $\til{p}$ is real stable.
\end{theorem}
\begin{proof}
    First, suppose that $\til{p}$ is real stable. We wish to show that for every $z\in\RR^{n+1}$, $p(te_0 + z)$ is real rooted. We note that $M$ is rank $n+1$, so there exists some $x\in\RR^{2n}$ such that $z = Mx$. Moreover, $M\ones = 2ne_0$, where $\ones$ denotes the vector of all ones. Then,
    \[
        p(te_0+z) = p\left(\frac{t}{2n}M\ones + Mx\right) = \til{p}\left(\frac{t}{2n}\ones + x\right).
    \]
    Since $\til{p}$ is real stable and $\frac{1}{2n}\ones\in\RR^{2n}_{>0}$, then this polynomial is real rooted, so $p$ is hyperbolic with respect to $e_0$, as desired.

    Now suppose that $p$ is hyperbolic with respect to $e_0$. By \cref{cor:hyp-wrt-simplicial}, $p$ is hyperbolic with respect to $K = \cone\{e_0 \pm \epsilon e_i\}$. To show that $\til{p}$ is stable, we wish to show that for any $v\in\RR^{2n}_{> 0}$ and $x\in\RR^{2n}$, $\til{p}(tv + x)$ is real rooted. Indeed, by definition,
    \[\til{p}(tv+x) = p(t(Mv) + Mx).\]
    Since $v\in\RR^{2n}_{> 0}$, then $Mv \in K$, and so $p(t(Mv) + Mx)$ is real rooted. Therefore, $\til{p}$ is stable, as desired.
\end{proof}

We are now ready to show that testing stability of cubics is coNP-hard. We require the following result of Nesterov:
\begin{theorem}[{\cite[Theorem 4]{RePEc:cor:louvco:2003071}}, {\cite[Theorem 5.2]{saunderson2019certifying}}]\label{thm:cubic-poly-max-clique}
    Let $G=(V,E)$ be a simple graph and let $\omega(G)$ be the size of a largest clique in $G$. Define
    \[q_G(x,y) = \sum_{(i,j)\in E} x_ix_jy_{ij}.\]
    Then
    \[\max_{\lVert x\rVert^2 +\lVert y\rVert^2 = 1} q_G(x,y) = \sqrt{\frac{2}{27}}\sqrt{1-\frac{1}{\omega(G)}}.\]
\end{theorem}

Since we want the cubic we are testing to have rational coefficients, we first need a quick technical lemma.
\begin{lemma}\label{lem:le-k-iff-le-ratl}
    Let $G$ be a graph on $n$ vertices, and let $\omega(G)$ denote its clique number. Let
    \[
        q_G(x,y) = \sum_{ij \in E} x_ix_jy_{ij}.
    \]
    For any integer $k \leq n$, there exists a rational number $\ell(k)$, computable in polynomial time in $n$, such that $\omega(G) \leq k$ if and only if
    \[
        \max_{\lVert x\rVert^2 + \lVert y\rVert^2 = 1} q_G(x,y) \leq \ell(k).
    \]
\end{lemma}
\begin{proof}
    For ease of notation, let $a(k) = \sqrt{\frac{2}{27}}\sqrt{1-\frac{1}{k}}$. By \cref{thm:cubic-poly-max-clique}, we know that $\max_{\lVert(x,y)\rVert^2 = 1} q_G(x,y) = a(\omega(G))$. Moreover, $a(k)$ is an increasing function, so $\max_{\lVert(x,y)\rVert^2=1} q_G(x,y) \leq a(k)$ if and only if $\omega(G) \leq k$.

    Define
    \[
        \ell(k) = \frac{\left\lceil 8n^2 a(k)\right\rceil}{8n^2}.
    \]
    We note that $\ell(k)$ is rational. Moreover, since $a(k) < \sqrt{\frac{2}{27}}$ for all $k \geq 1$, its numerator and denominator are both polynomial in $n$, so $\ell(k)$ can be computed in polynomial time.

    Since $\ell(k) \geq a(k)$, one direction is immediate. Namely, if $\omega(G) \leq k$, then
    \[
        \max_{\lVert x\rVert^2 + \lVert y\rVert^2 = 1}q_G(x,y) \leq a(k) \leq \ell(k).
    \]
    
    For the converse, suppose that $\omega(G) > k$. Since $\omega(G)$ is always an integer, this implies that $\omega(G) \geq k+1$, and so $\max_{\lVert (x,y)\rVert^2 =1} q_G(x,y) \geq a(k+1)$.
    
    In order to complete the proof, it remains to show that $\ell(k) < a(k+1)$. Indeed, for any $1 < x \leq n$,
    \[
        a'(x) = \sqrt{\frac{2}{27}}\frac{1}{2\sqrt{x^3(x-1)}} > \frac{1}{\sqrt{54}}\frac{1}{x^2} \geq \frac{1}{\sqrt{54}}\frac{1}{n^2} > \frac{1}{8n^2},
    \]
    so
    \[
        a(k+1) = a(k) + \int_{k}^{k+1}a'(x)dx \geq a(k) + \frac{1}{8n^2}.
    \]
    Since $\lceil 8n^2 a(k)\rceil < 8n^2 a(k) + 1$, it follows that $\ell(k) < a(k) + \frac{1}{8n^2} \leq a(k+1)$.
    
    Thus, if $\omega(G) > k$, then $\max_{\lVert(x,y)\rVert^2=1} q_G(x,y) > \ell(k)$, and so $\omega(G) \leq k$ if and only if $\max_{\lVert(x,y)\rVert^2=1}q_G(x,y) \leq \ell(k)$, as desired.
\end{proof}

\begin{theorem}\label{thm:cubic-stable-hard}
    Given a homogeneous cubic polynomial $p$ with rational coefficients, it is coNP-hard to decide if $p$ is stable.
\end{theorem}
\begin{proof}
    Given a graph $G$ and integer $k$, it is coNP-hard to decide if $\omega(G) \leq k$ \cite{Karp1972}. Given $G$ and $k$, define $\ell = \ell(k)$ as in \cref{lem:le-k-iff-le-ratl}, and consider the cubic polynomial
    \[
        p(x_0,x,y) = x_0^3 - 3x_0\left(\lVert x\rVert^2 + \lVert y\rVert^2\right) + 2\frac{q_G(x,y)}{\ell}.
    \]
    By \cref{thm:hyp-iff-max-le-1}, this polynomial is hyperbolic with respect to $e_0$ if and only if
    \begin{align*}
        \max_{\lVert (x,y)\rVert^2 = 1}\left\lvert\frac{q_G(x,y)}{\ell}\right\rvert \leq 1 &\Leftrightarrow \max_{\lVert (x,y)\rVert^2 = 1} q_G(x,y) \leq \ell\\
            &\Leftrightarrow \omega(G) \leq k,
    \end{align*}
    so it is coNP-hard to decide if this polynomial (with rational coefficients) is hyperbolic with respect to $e_0$ and therefore whether $\omega(G) \leq k$. By \cref{thm:hyp-iff-stable}, deciding whether $p$ is hyperbolic with respect to $e_0$ is equivalent to checking stability of $\til{p} = p(Mx)$.

    By choosing $\epsilon$ to be rational, then $\til{p}$ will have rational coefficients, so if we could check stability of homogeneous cubics with rational coefficients, we could also check hyperbolicity of $p$ with respect to $e_0$.

    Since $N$ is the largest norm of a coefficient of $q(x)$, it is already encoded in the input, so computing a positive rational number $\epsilon < \min(\frac{1}{2Nn^3}, \frac{1}{2})$ takes polynomial time, as does performing the necessary change of variables by $M$.

    Therefore, we have a polynomial time reduction, and so deciding stability for a homogeneous cubic with rational coefficients is coNP-hard.
\end{proof}

\begin{example}
    Let $G$ be the path graph on three vertices, and let $k = 2$. We construct $q_G(x_1,x_2,x_3,y_{12},y_{23}) = x_1x_2y_{12} + x_2x_3y_{13}$. Then $a(k) = \sqrt{1/{27}}$, and $\ell(k) = \frac{\lceil 8(25)/\sqrt{27}\rceil}{8(25)} = \frac{39}{200}.$ Following the construction in \cref{thm:cubic-stable-hard} gives
    \[p(x_0,x_1,x_2,x_3,y_{12},y_{23}) = x_0^3 - 3x_0(x_1^2 + x_2^2 + x_3^2 + y_{12}^2 + y_{23}^2) + \frac{400}{39}(x_1x_2y_{12} + x_2x_3y_{23}).\]

    Thus, $N = \frac{400}{39}$, so we may take $\epsilon = 10^{-4} < \frac{1}{2\left(\frac{400}{39}\right) 5^3} = \frac{39}{50000}.$ Finally, we construct $\til{p} \in \RR[z_1,z_2,z_3,z_{12},z_{23},w_1,w_2,w_3,w_{12},w_{13}]$, where $z_i \mapsto x_0 + 10^{-4} x_i$, $z_{ij} \mapsto x_0 + 10^{-4}y_{ij}$, $w_i\mapsto x_0 - 10^{-4}x_i$, and $w_{ij}\mapsto x_0 - 10^{-4}y_{ij}$. Then,
    \begin{multline*}
        \til{p}(z,w) = \left(z_1 + w_1 + z_2 + w_2 + z_3 + w_3 + z_{12} + w_{12} + z_{23} + w_{23}\right)^3\\
        - 3\cdot 10^{-8}\left(z_1 + w_1 + z_2 + w_2 + z_3 + w_3 + z_{12} + w_{12} + z_{23} + w_{23}\right)\lVert z-w\rVert^2\\
        + \frac{400}{39}\cdot10^{-12}\left((z_1-w_1)(z_2-w_2)(z_{12}-w_{12}) + (z_2-w_2)(z_3-w_3)(z_{23}-w_{23})\right).
    \end{multline*}
    Since $\omega(G) \leq 2$, $\til{p}$ is real stable.
\end{example}

As an easy corollary to \cref{thm:cubic-stable-hard}, it is also coNP-hard to test for stability in higher degrees.
\begin{corollary}\label{cor:stability-conp-hard}
    Fix a degree $d \geq 3$. Given a homogeneous polynomial $p$ with rational coefficients of degree $d$, it is coNP-hard to decide if $p$ is stable.
\end{corollary}
\begin{proof}
    Given a homogeneous cubic $p\in\RR[z_1,\dots,z_n]$, we introduce a new variable $y$ and consider the polynomial $q = y^{d-3}p\in\RR[y,z_1,\dots,z_n]$. Then $q$ is a homogeneous polynomial of degree $d$, and $q$ is stable if and only if $p$ is stable.

    Thus, if we can test stability in degree $d \geq 3$, we can also test stability of homogeneous cubics, so stability is coNP-hard for all $d \geq 3$.
\end{proof}

\section{Lorentzianity}\label{sec:complete-log concavity}
While all homogeneous stable polynomials are Lorentzian, the converse is not true. In fact, while deciding if a polynomial is real stable is coNP-hard, deciding complete log concavity for polynomials of fixed degree $d$ can be done in polynomial time in the number of variables $n$.

The key is the following result from \cite{anari2018logconcave3}. We say a polynomial $f\in\RR[z_1,\dots,z_n]$ is \emph{indecomposable} if it cannot be written as $f = f_1 + f_2$, where $f_1,f_2$ are nonzero polynomials in disjoint sets of variables.

\begin{theorem}[{\cite[Theorem 3.2]{anari2018logconcave3}}]\label{thm:clc-conditions}
    Let $f\in\RR[z_1,\dots,z_n]$ be a homogeneous polynomial of degree $d \geq 2$ with nonnegative coefficients. Then $f$ is completely log concave if and only if the following two conditions hold:
    \begin{enumerate}[(i)]
        \item For all $\alpha\in\ZZ^n_{\geq 0}$ with $|\alpha|\leq d-2$, the polynomial $\partial^\alpha f$ is indecomposable.
        \item For all $\alpha\in\ZZ^n_{\geq 0}$ with $|\alpha| = d-2$, the quadratic polynomial $\partial^\alpha f$ is log concave over $\RR^n_{\geq 0}$.
    \end{enumerate}
\end{theorem}

This means that complete log concavity can be decided in polynomial time.
\begin{proposition}\label{prop:clc-poly-time}
    For fixed degree $d$, there is an algorithm, which runs in time $O(n^{d+1})$, that decides if a homogeneous polynomial $f\in\RR_{\geq 0}[z_1,\dots,z_n]_d$ is completely log concave on $\RR^n_{\geq 0}$.
\end{proposition}
\begin{proof}
    Since $d$ is fixed, there are $O(n^{d-2})$ partial derivatives to check. In order to check if a polynomial is indecomposable, we form a graph with vertices $\{i \in [n] : \dl_i f \neq 0\}$ and add edge $(i,j)$ if and only if there is a monomial of $f$ that contains both $x_i$ and $x_j$. Then $f$ is indecomposable if and only if this graph is connected, and we can construct this graph and check if it is connected in time $O(n^2)$. Since there are $O(n^{d-2})$ derivatives to check, each of which takes $O(n^2)$ time, we can check the first condition of \cref{thm:clc-conditions} in time $O(n^d)$.
    
    Moreover, if $|\alpha| = d-2$, then $\partial^\alpha f$ is quadratic, so the entries of $\nabla^2 \partial^\alpha f(z)$ do not depend on $z$, and $\partial^\alpha f$ is log concave if and only if $\nabla^2 \partial^\alpha f$ has exactly one positive eigenvalue. We can compute the signs of the eigenvalues of an $n\times n$ matrix in $O(n^3)$ time using Gaussian elimination \cite{beightler1966lagrange}, so it follows that checking whether $\partial^\alpha f$ is log concave can be accomplished in time at most $O(n^3)$. There are $O(n^{d-2})$ derivatives to check, each of which takes time $O(n^3)$, so we can check the second condition in time $O(n^{d+1})$.

    Therefore, we can check both conditions of \cref{thm:clc-conditions} in time at most $O(n^{d+1})$, as desired.
\end{proof}

\section{Computational Complexity of Log Concavity}\label{sec:quartic-log-concavity}

In \cref{sec:complete-log concavity}, we showed that we can decide if a homogeneous polynomial is completely log concave in polynomial time in the number of variables. However, unlike complete log concavity, we cannot reduce checking if a polynomial is log concave to checking its quadratic derivatives, and in fact deciding if a polynomial is log concave is coNP-hard in degrees $d \geq 4$.

\begin{theorem}\label{thm:quartic-lc-hard}
    It is coNP-hard to decide if a homogeneous polynomial of degree four with nonnegative coefficients is log concave on $\RR^n_{\geq 0}$.
\end{theorem}

To prove this, we first give a reduction from checking if a homogeneous quartic is convex on $\RR^n_{\geq 0}$. Then, we refine the main result from \cite{Ahmadi_2011} to show that it is coNP-hard to decide if a polynomial is convex on the nonnegative orthant.

\begin{lemma}\label{lem:lc-iff-cvx-on-nonneg}
    Let $f\in\RR[x_1,\dots,x_n]_4$ be a homogeneous quartic polynomial, and let $N > 0$ be at least as large as the largest coefficient of $f$. Let
    \[g(x_1,\dots,x_n,z) = N\left(z + \sum_{i=1}^n x_i\right)^4 - f(x_1,\dots x_n) \in \RR[x_1,\dots,x_n,z].\]
    Then $g$ is a homogeneous quartic with nonnegative coefficients, and $f$ is convex on $\RR^n_{\geq 0}$ if and only if $g$ is log concave.
\end{lemma}
\begin{proof}
    First, we claim that $g$ has nonnegative coefficients. Indeed, every degree four monomial in $x_1,\dots,x_n$ appears in $N(z + x_1 + \dots +x_n)^4$ with coefficient at least $N$. Since we chose $N$ to be at least as large as the largest coefficient of $f$, then the coefficient of every monomial of $g$ is at least zero, so $g$ has nonnegative coefficients.

    Next, we claim that if $f$ is convex on $\RR^n_{\geq 0}$, then $g$ is log concave. We note that by construction,
    \begin{equation}
        \nabla^2 g(x,z) = 12N(z + x_1 + \dots + x_n)^2 \ones\ones^T + \left(\begin{array}{ccc|c}
             & & & 0 \\
             & -\nabla^2 f(x) & & \vdots\\
             & & & 0\\\hline
             0 & \dots & 0 & 0
        \end{array}\right),
    \end{equation}
    where $\ones$ denotes the vector of all ones. Since $g$ is homogeneous, $g$ is log concave at $(x,z)\in\RR^{n+1}_{\geq 0}$ if and only if $\nabla^2 g(x,z)$ has at most one positive eigenvalue \cite[Proposition 2.33]{brändén2022lorentzian}. If $f$ is convex on $\RR^n_{\geq 0}$, then $-\nabla^2 f(x)\preceq 0$, and since the first term is rank one, then $\nabla^2 g(x,z)$ has at most one positive eigenvalue, so $g$ is log concave, as desired.

    Now, suppose that $f$ is not convex on $\RR^n_{\geq 0}$. Then we claim that $g$ is not log concave. Indeed, suppose that $f$ is not convex at $x\in\RR^n_{\geq 0}$, and let $Q = \nabla^2 g(x,1)$. We claim that in order to show that $g$ is not log concave at $(x,1)\in\RR^{n+1}_{\geq 0}$, it suffices to show that $Q \succ 0$ on a two-dimensional subspace. Indeed, if $Q$ has at most one positive eigenvalue, then there is a subspace of codimension one on which $Q\preceq 0$. Specifically, if $v_1,\dots,v_n$ are eigenvectors of $Q$ with $\lambda_1,\dots,\lambda_n \leq 0$, then $Q\preceq 0$ on $\spanset\{v_1,\dots,v_n\}$. If $Q \succ 0$ on a two-dimensional subspace, then it cannot be negative semidefinite on a subspace of codimension one, so $Q$ must have at least two positive eigenvalues.

    Since $f$ is not convex at $x$, then $\nabla^2 f(x) \not\succeq 0$, so there exists some $v\in\RR^n\setminus\{0\}$ and $\lambda > 0$ such that $-\nabla^2 f(x) v = \lambda v$. We claim that $Q \succ 0$ on $L = \spanset\left\{\begin{pmatrix}
        \vec{v} \\ 0
    \end{pmatrix}, \begin{pmatrix}
        \vec{0} \\ 1
    \end{pmatrix}\right\} \subseteq \RR^{n+1}$.

    Let $w = \begin{pmatrix}
        \alpha v \\ \beta
    \end{pmatrix} \in L$. Then
    \begin{align*}
        w^T Q w &= 12N(1 + x_1 + \dots + x_n)^2 (\ones^T w)^2 + \alpha^2 v^T (-\nabla^2 f(x)) v\\
            &= 12 N(1+x_1+\dots+x_n)^2 (\ones^T w)^2 + \alpha^2 \lambda \lVert v\rVert^2.
    \end{align*}
    Since both terms are nonnegative, $w^T Q w \geq 0$, with equality if and only if $12 N(1+x_1+\dots+x_n)^2 (\ones^T w)^2 = 0$ and $\alpha^2 \lambda \lVert v\rVert^2 = 0$.
    
    Since $\lambda > 0$ and $v \neq 0$, the latter happens if and only if $\alpha = 0$. Thus, if $w^TQ w = 0$, then $\alpha = 0$, so $w = \begin{pmatrix}
        0 \\ \beta
    \end{pmatrix}$ and 
    \[0 = \begin{pmatrix}
        0 \\ \beta
    \end{pmatrix}^T Q \begin{pmatrix}
        0 \\ \beta
    \end{pmatrix} = 12 N (1 + x_1 + \dots + x_n)^2 \beta^2.\]
    
    Since $x\in\RR^n_{\geq 0}$, then $x_1+\dots+x_n+1 > 0$. Additionally, we chose $N > 0$, so this implies that $\beta = 0$. Thus, if $w\in L$, then $w^T Q w \geq 0$ with equality if and only if $w = 0$, so $Q \succ 0$ on $L$.

    Therefore, $\nabla^2 g(x,1)$ is positive definite on a two-dimensional subspace, so $\nabla^2 g(x,1)$ has at least two positive eigenvalues, so $g$ is not log concave at $(x,1)$.
    
    Thus, if $f$ is not convex on $\RR^n_{\geq 0}$, then $g$ is not log concave, as desired.
\end{proof}

Next, we show that it is coNP-hard to decide if a quartic polynomial is convex on $\RR^n_{\geq 0}$. We use the construction from \cite{Ahmadi_2011} and refine their argument to show that for a graph $G$ and integer $k$, we can construct a quartic polynomial $f$ so that $f$ is convex on the nonnegative orthant if and only if $G$ has no clique of size larger than $k$.

\begin{theorem}[{\cite{Ahmadi_2011}, Theorem 2.3}]\label{thm:ahmadi-convex-iff-psd}
    Given a biquadratic form $b(x;y) = \sum_{i,j=1}^n \alpha_{ij} x_ix_jy_iy_j$, define the $n\times n$ polynomial matrix $C(x,y)$ by setting $[C(x,y)]_{ij} = \frac{\partial^2 b(x;y)}{\partial x_i \partial y_j}$ and let $\gamma$ be the largest coefficient, in absolute value, of any monomial present in some entry of the matrix $C(x,y)$. Let $f$ be the form given by
    \begin{equation}\label{eqn:quartic-convex}
        f(x,y) = b(x;y) + \frac{n^2 \gamma}{2}\left(\sum_{i=1}^n x_i^4 + \sum_{i=1}^n y_i^4 + \sum_{\substack{i,j=1,\dots,n \\ i < j}} x_i^2 x_j^2 + \sum_{\substack{i,j=1,\dots,n \\ i < j}} y_i^2 y_j^2\right).
    \end{equation}
    Then $b(x;y)$ is PSD if and only if $f(x,y)$ is convex.
\end{theorem}

In \cite{Ahmadi_2011}, given a graph $G$ and integer $k$, they use a result from \cite{ling-biquadratic} to construct a biquadratic form that is PSD if and only if $\omega(G) \leq k$, which implies that deciding convexity of quartics is coNP-hard.
\begin{lemma}[\cite{Ahmadi_2011,ling-biquadratic}]
    Let $G = ([n],E)$ be a graph and let $1 \leq k \leq n$. Define the biquadratic form
    \begin{equation}\label{eqn:graph-biquadratic}
        b_G(x;y) = -2k \sum_{ij\in E} x_ix_jy_iy_j - (1-k)\left(\sum_{i=1}^n x_i^2\right)\left(\sum_{i=1}^n y_i^2\right).
    \end{equation}
    Then $b_G(x;y)$ is PSD if and only if $\omega(G) \leq k$.
\end{lemma}

We use the same biquadratic form and construction, and refine the result to show that it is coNP-hard to test if a quartic is convex on the nonnegative orthant, rather than on all of $\RR^n$.
\begin{lemma}\label{lem:cvx-on-nonneg-iff-clique}
    Let $G = ([n], E)$ be a graph, and let $1 \leq k \leq n$. Define the biquadratic form $b_G(x;y)$ as in \cref{eqn:graph-biquadratic},
    and construct $f$ as in \cref{thm:ahmadi-convex-iff-psd}. Then $f$ is convex on $\RR^{2n}_{\geq 0}$ if and only if $\omega(G) \leq k$.
\end{lemma}
\begin{proof}
    One direction is immediate from \cite{Ahmadi_2011}. Namely, if $\omega(G) \leq k$, then $b_G(x;y) \succeq 0$. Then by the result from \cite{Ahmadi_2011}, $f$ is convex, so in particular, it is convex on $\RR^{2n}_{\geq 0}$.

    Now, suppose that $\omega(G) > k$. Following \cite{ling-biquadratic} and \cite{Motzkin_Straus_1965}, let $\mathcal{C} \subseteq [n]$ be a clique of size $\omega(G)$ in $G$, and define $x^*, y^* \in \RR^n_{\geq 0}$ by
    \[
        x_i^* = y_i^* = \begin{cases}
            \frac{1}{\sqrt{\omega(G)}}, & i \in \mathcal{C}\\
            0, & \text{otherwise}
        \end{cases}.
    \]
    We claim that $f$ is not convex at $(x^*, 0) \in \RR^{2n}_{\geq 0}$.
    
    We note that
    \[
        b_G(x^*; y^*) = -2k \binom{\omega(G)}{2} \frac{1}{\omega(G)^2} - (1-k) = \frac{k}{\omega(G)} - 1 < 0.
    \]
    
    Moreover, we note that for any $x\in\RR^n$, $\frac{\dl^2}{\dl y_i\dl y_j} f(x, 0) = \frac{\dl^2}{\dl y_i\dl y_j}b_G(x,0)$, and $\frac{\dl^2}{\dl y_i\dl y_j}b_G(x,y)$ does not depend on $y$, so $\frac{\dl^2}{\dl y_i\dl y_j}b_G(x,0) = \frac{\dl^2}{\dl y_i\dl y_j} b_G(x,y^*)$. Thus, since $b_G$ is homogeneous of degree $2$ in the $y$ variables,
    \[\begin{pmatrix}
        0 \\ y^*
    \end{pmatrix}^T \nabla^2 f(x^*, 0) \begin{pmatrix}
        0 \\ y^*
    \end{pmatrix} = \sum_{i,j=1}^n y_i^*y_j^* \frac{\dl^2}{\dl y_i \dl y_j}b_G(x^*, y^*) = 2b_G(x^*, y^*) < 0.\]
    Therefore, $\nabla^2 f(x^*, 0) \not\succeq 0$, and so $f$ is not convex on $\RR^{2n}_{\geq 0}$, as desired.
\end{proof}

Finally, we combine \cref{lem:lc-iff-cvx-on-nonneg} and \cref{lem:cvx-on-nonneg-iff-clique} to show that deciding log concavity is coNP-hard.
\begin{proof}[Proof of \cref{thm:quartic-lc-hard}]
    Let $G = ([n],E)$ be a graph on $n$ vertices, and let $1 \leq k \leq n$. Construct $b_G(x;y)$ as in \cref{eqn:graph-biquadratic} and $f$ as in \cref{eqn:quartic-convex}. Then we have a quartic polynomial in $2n$ variables, and by \cref{lem:cvx-on-nonneg-iff-clique}, $f$ is convex on $\RR^{2n}_{\geq 0}$ if and only if $\omega(G) \leq k$. Moreover, constructing $f$ takes polynomial time in $n$ \cite{Ahmadi_2011}.

    Then, using \cref{lem:lc-iff-cvx-on-nonneg}, we get a homogeneous quartic polynomial $g$ in $2n+1$ variables with nonnegative coefficients so that $g$ is log concave if and only if $f$ convex on $\RR^{2n}_{\geq 0}$. Thus, $g$ is log concave if and only if $\omega(G) \leq k$, so it just remains to show that constructing $g$ from $f$ takes polynomial time.

    Indeed, to get from $f$ to $g$, we add a polynomial with $\binom{n+4}{4}$ monomials, each with coefficient at most $24N$. The size of $N$ is polynomial in our input, since we may choose it to be the largest coefficient of $f$ and the coefficients of $f$ are part of our input.

    Thus, our construction runs in polynomial time, so it is coNP-hard to decide if a homogeneous quartic polynomial is log concave.
\end{proof}

\begin{corollary}\label{cor:lc-hard-higher-deg}
    For fixed degree $d \geq 4$, it is coNP-hard to decide if a homogeneous polynomial of degree $d$ with nonnegative coefficients is log concave.
\end{corollary}
\begin{proof}
    We give a reduction from the problem of deciding log concavity of homogeneous quartic polynomials. Let $f \in \RR_{\geq 0}[x_1,\dots,x_n]$ be a homogeneous quartic, and define $g = z^{d-4}f(x)\in\RR_{\geq 0}[x_1,\dots,x_n,z]$. We note that $g$ is a homogeneous polynomial of degree $d$ with nonnegative coefficients, and we claim that $g$ is log concave if and only if $f$ is log concave.
    
    Indeed, we note that $\log g = (d-4)\log(z) + \log f$, so
    \[
        \nabla^2 \log g(x,z) = (d-4)\left(\begin{array}{ccc|c}
             0 & \dots & 0 & 0 \\
             \vdots & \ddots & \vdots & \vdots\\
             0 & \dots & 0 & 0\\\hline
             0 & \dots & 0 & \frac{-1}{z^2}
        \end{array}\right) + \left(\begin{array}{ccc|c}
             & & & 0 \\
             & \nabla^2\log f(x) & & \vdots\\
             & & & 0\\\hline
             0 & \dots & 0 & 0
        \end{array}\right).
    \]
    Then $\nabla^2\log g(x,z) \preceq 0$ if and only if $\nabla^2 \log f(x) \preceq 0$, so $g$ is log concave if and only if $f$ is log concave.

    This construction takes polynomial time and only increases the size of the input by a polynomial amount. Therefore, since it is coNP-hard to decide if a homogeneous quartic is log concave, it is also coNP-hard to decide if a homogeneous polynomial of degree $d \geq 4$ is log concave, as desired.
\end{proof}

\section{Log Concavity For Cubics}\label{sec:log concavity-cubic}
For a polynomial $f$ of degree $d \geq 3$, we know that it is coNP-hard to decide if $f$ is real stable (\cref{sec:stability}) or hyperbolic with respect to a given direction \cite{saunderson2019certifying}. However, our results in \cref{sec:quartic-log-concavity} only apply to polynomials of degree four or more.

In this section, we examine the reason for this gap by studying log concavity of cubic polynomials from two different perspectives. First, we examine the weaker property of directional log concavity, where we restrict ourselves to testing whether a polynomial is log concave in a given direction, and show that testing this property is coNP-hard for homogeneous cubics. On the other hand, in \cref{sec:homogeneous-cubics}, we show that we can decide if a homogeneous cubic is log concave in the usual sense in polynomial time.

\subsection{Directional Log Concavity}\label{sec:directional-log concavity}
In this section, we prove that testing whether a homogeneous cubic is log concave in a given direction is coNP-hard. The goal of this section is to show that it is coNP-hard to test directional log concavity of cubic polynomials.

First we give a formal definition of directional log concavity.
\begin{definition}
    We say that a polynomial $f\in\RR_{\geq 0}[x_1,\dots,x_n]$ is \emph{log concave in direction $v\in\RR^n$} if for all $x\in\RR^n_{\geq 0}$, $D_v^2\log f(x) \leq 0$.
\end{definition}

To show that testing directional log concavity is coNP-hard, we first relate testing directional log concavity to maximizing a homogeneous cubic.
\begin{theorem}
    \label{thm:cubic-log-concave-in-direction}
    Let $q(x)\in\RR_{\geq 0}[x_1,\dots,x_n]_3$. Then
    \[
        f(x,z) = z^3 + 3\lVert x\rVert^2 z + 2q(x)
    \]
    is log concave in the $z$ direction at all $(x,z) \in \RR^{n+1}_{\geq 0}$ if and only if $\max_{\lVert x\rVert = 1} q(x) \leq 1$.
\end{theorem}

To prove this, we first characterize log concavity for depressed cubics.
\begin{lemma}\label{lem:depressed-cubic-lc}
    Let $f(z) = z^3 + bz + c \in\RR_{\geq 0}[z]$. Then $f$ is log concave on $\RR_{\geq 0}$ if and only if $4b^3 \geq 27c^2$.
\end{lemma}
\begin{proof}
    Define $g(z) = f^2 \frac{d^2}{dz^2}\log f = f\cdot f'' - (f')^2$. Then $f$ is log concave at $z$ if and only if $g(z) \leq 0$.

    By direct computation,
    \[g(z) = (z^3+bz+c)(6z) - (3z^2 + b)^2 = -3z^4 + 6cz -b^2.\]
    Then
    \[g'(z) = -12z^3 + 6c,\]
    so $g$ is minimized at $z = \left(\frac{c}{2}\right)^{1/3} \geq 0$, and
    \[g\left(\left(\frac{c}{2}\right)^{1/3}\right) = 9\left(\frac{c}{2}\right)^{4/3} - b^2.\]
    Thus, $f$ is log concave if and only if $b^2 \geq 9\left(\frac{c}{2}\right)^{4/3}$, which happens if and only if $4b^3 \geq 27c^2$, as desired.
\end{proof}

\begin{proof}[Proof of \cref{thm:cubic-log-concave-in-direction}]
    Fix $x \in \RR^n_{\geq 0}$. By \cref{lem:depressed-cubic-lc}, $f$ is log concave (as a function of $z$) at $x$ if and only if $\lVert x\rVert^6 \geq q(x)^2$. Since $q$ is homogeneous of degree three, this happens if and only if $q(\frac{x}{\lVert x \rVert})^2 \leq 1$.

    Since $\{x\in\RR^n : \lVert x \rVert = 1\}$ is compact, we know that $q$ achieves its maximum on this set at some $x^*$. Moreover, since $q$ has nonnegative coefficients, we may assume without loss of generality that $x^*\in\RR^n_{\geq 0}$. If $q(x^*) > 1$, then $q(x^*)^2 > \lVert x^*\rVert^6$, and so $f$ is not log concave in the $z$ direction at $(x^*,z^*)$ for some $z^* \geq 0$.

    Conversely, if $q(x^*) \leq 1$, then for any $x \in \RR^n_{\geq 0}$, $q(x) \leq \lVert x\rVert^3$, and so $f$, as a function of $z$, satisfies the conditions of \cref{lem:depressed-cubic-lc} and is therefore log concave in the $z$ direction at all $(x,z)\in\RR^{n+1}_{\geq 0}$.
\end{proof}

By the same argument as we used in \cref{thm:cubic-stable-hard}, it is coNP-hard to decide if $\max_{\lVert x\rVert=1} q(x) \leq 1$, so we get the following hardness result for directional log concavity.

\begin{theorem}
    It is coNP-hard to determine if a homogeneous cubic is log concave on the positive orthant in a given direction.
\end{theorem}

\subsection{Log Concavity for Homogeneous Cubics}\label{sec:homogeneous-cubics}
Although it is coNP-hard to test whether a homogeneous cubic is log concave in a particular direction, we can test whether one is log concave on $\RR^n_{>0}$ in polynomial time.

Since we can test complete log concavity in polynomial time (\cref{prop:clc-poly-time}), it suffices to show that a homogeneous cubic is log concave if and only if it is completely log concave.
\begin{theorem}
    Let $f\in\RR_{\geq 0}[x_1,\dots,x_n]_3$ be a homogeneous cubic. Then $f$ is log concave on $\RR^n_{> 0}$ if and only if it is completely log concave with respect to $\RR^n_{> 0}$.
\end{theorem}
\begin{proof}
    One direction is immediate: if $f$ is completely log concave then it is by definition also log concave.

    It remains to show that if $f$ is log concave, then it is also completely log concave. First, since $f$ is log concave, then it is indecomposable \cite{anari2018logconcave3}. Therefore, in order to show it is completely log concave, it suffices to show that $\partial_1 f,\dots,\partial_n f$ are all log concave \cite[Lemma 3.4]{anari2018logconcave3}.
    
    Since $f$ is homogeneous of degree 3, by Euler's identity,
    \[
        3f(x) = \sum_{i=1}^nx_i\partial_if(x).
    \]

    Let $E_{ij}$ denote the $n\times n$ matrix with a one in the $(i,j)$th entry and zeros elsewhere. Using the product rule for the Hessian,
    \begin{align*}
        3\nabla^2f(x) &= \sum_{i=1}^n x_i\nabla^2\partial_if(x) + \sum_{i=1}^n(e_i(\nabla \partial_if(x))^T + (\nabla \partial_if(x)) e_i^T)\\
            &= \sum_{i=1}^n x_i\nabla^2\partial_if(x) + \sum_{i=1}^n\sum_{j=1}^n\partial^2_{ij}f(x) E_{ij} + \partial^2_{ij}f(x)E_{ji}\\
            &= \sum_{i=1}^n x_i\nabla^2\partial_if(x) + 2\nabla^2 f(x),
    \end{align*}
    so
    \[
        \nabla^2 f(x) = \sum_{i=1}^n x_i\nabla^2\partial_if(x).
    \]

    Since $\partial_i f$ is a homogeneous polynomial of degree two, then $\nabla^2 \partial_i f$ does not depend on $x$. Let $M_i = \nabla^2\partial_i f$, and let $Q(x) = \sum_{i=1}^n x_iM_i$, so $\nabla^2f(x) = Q(x)$. It remains to show that $\lambda_2(M_i) \leq 0$ for all $i=1,\dots,n$.

    Since $f$ is log concave by assumption, then $\lambda_2(Q(x)) \leq 0$ for all $x\in\RR^n_{> 0}$. Since eigenvalues are continuous functions in the entries of a matrix, then $\lambda_2(Q(x)) \leq 0$ for all $x\in\RR^n_{\geq 0}$. Since $M_i = Q(e_i)$, this implies that $\lambda_2(M_i) \leq 0$, so $\partial_i f$ is log concave, as desired.
\end{proof}

Since we can decide if a polynomial is completely log concave in polynomial time, this has immediate implications for testing log concavity of homogeneous cubics.
\begin{corollary}
    There is a polynomial time algorithm to decide if a homogeneous cubic polynomial is log concave on $\RR^n_{\geq 0}$.
\end{corollary}

\section{Conclusion and Future Work}
This paper discusses stable, completely log-concave, and log-concave polynomials. In particular, we show that it is coNP-hard to test if polynomials are stable or log concave. While stable and log-concave polynomials are useful tools, this implies that (assuming P $\neq$ NP), it is intractible to test if a given polynomial is stable or log concave, unless additional structure is present.

Throughout this paper, we used reductions from the max clique problem to show that testing our desired properties is coNP-hard. However, our reductions did not use the full strength of max clique: the max clique problem is not only NP-hard, but also NP-hard to approximate. This hardness of approximation should allow us to strengthen our results, but the exact form of this conjectured stronger result remains open.

\paragraph{Acknowledgements} I would like to thank Cynthia Vinzant for many helpful discussions and comments. I would also like to thank Thomas Rothvoss and Rekha Thomas for their thoughtful suggestions, as well as Amir Ali Ahmadi for his feedback. The author was partially supported by NSF grant No.~DMS-2153746.

\bibliographystyle{plainnat}
\bibliography{ref}
\end{document}

%% file: symbol_defns.tex
\def\CC{{\mathbb C}}

\def\NN{{\mathbb N}}

\def\RR{{\mathbb R}}

\def\ZZ{{\mathbb Z}}

\def\e{\epsilon}

\def\e{\varepsilon}
\renewcommand{\phi}{\varphi}

\def\ones{\mathbbm{1}}

\newcommand{\dl}{\partial}

\newcommand{\til}{\widetilde}

\def\cone{\operatorname {cone}}

\def\spanset{\operatorname {span}}

\renewcommand{\Re}{\operatorname{Re}}
\renewcommand{\Im}{\operatorname{Im}}


%% file: lorentz_cone.tex
\tdplotsetmaincoords{80}{120}
\begin{tikzpicture}[axis/.style={->,black,thick}, tdplot_main_coords]
    \def\x{1.4}
    \def\y{3.5}
    \def\R{\x+0.02}
    \def\yc{\y+1}
    \def\e{0.6}

    \draw[axis] (0,0,0) -- (0,\yc,0);
    \draw[axis] (0,0,0) -- (\yc,0,0);
    \draw[axis] (0,0,0) -- (0,0,\yc) node[anchor=west]{$e_0$};


    \coordinate(O) at (0,0,0);
    \draw[canvas is xy plane at z=\y, draw=black, left color=yellow,right color=yellow,middle
        color=yellow!60!white, fill opacity=0.2] (\tdplotmainphi:\R) 
    arc(\tdplotmainphi:\tdplotmainphi+180:\R) -- (O) --cycle;
    \draw[canvas is xy plane at z=\y, draw=black, left color=yellow,right color=yellow,middle
        color=yellow!60!white, fill opacity=0.2] (\tdplotmainphi:\R) 
    arc(\tdplotmainphi:\tdplotmainphi-180:\R) -- (O) --cycle;
    \begin{scope}[canvas is xy plane at z=\y]
        \draw[fill=red, opacity=0.3]
        (-\R,0) -- (0,\R) -- (O) --cycle;
        \draw[fill=red, opacity=0.3]
        (0,\R) -- (\R,0) -- (O) -- cycle;
        \draw[fill=red, opacity=0.3]
        (0,-\R) -- (\R,0) -- (O) -- cycle;
        \draw[fill=red, opacity=0.3]
        (0,-\R) -- (-\R,0) -- (O) -- cycle;
    \end{scope}
\end{tikzpicture}